\documentclass[a4paper, 10pt, conference]{ieeeconf}      

\IEEEoverridecommandlockouts                              
\usepackage{amssymb}

\usepackage[pdftex]{graphicx}
\usepackage{tikz}
\usepackage{amsmath}
\usepackage{framed}
\usepackage{bm}
\usepackage{cite}
\usepackage{comment}
\usepackage{algorithmic}
\usepackage{algorithm}
\usepackage{url}
\usepackage{booktabs}
\usepackage{array}

\usetikzlibrary{intersections,calc,arrows.meta}

\usepackage[sort]{cleveref}
\usepackage{mathtools}
\usepackage{subcaption}
\usepackage{multirow}
\usepackage{threeparttable}

\usepackage{pgfplots}

\usepackage{mathtools}
\mathtoolsset{showonlyrefs=true} 

\usepackage{amsfonts}
\usepackage{color}
\usepackage{cases}
\usepackage{enumerate}

\newtheorem{theorem}{Theorem}
\newtheorem{assumption}{Assumption}

\newtheorem{remark}{Remark}
\newtheorem{lemma}{Lemma}
\newtheorem{example}{Example}

\newtheorem{proposition}{Proposition}

\usepackage{blkarray}

\DeclareMathOperator{\diag}{diag}
\DeclareMathOperator{\tr}{tr}

\newcommand{\e}{\mathrm{e}}

\newcommand{\D}{\mathrm{d}}

\newcommand{\Real}{\mathbb{R}}

\newcommand{\bb}{\mathbb}

\def\BibTeX{{\rm B\kern-.05em{\sc i\kern-.025em b}\kern-.08em
    T\kern-.1667em\lower.7ex\hbox{E}\kern-.125emX}}

\title{\LARGE \bf Relationship Between Controllability Scoring and\\ Optimal Experimental Design}

\author{Kazuhiro Sato
\thanks{This work was supported by JST PRESTO, Japan, Grant Number JPMJPR25K4. }
\thanks{K. Sato is with Department of Mathematical Informatics, Graduate School of Information Science and Technology, The University of Tokyo, Tokyo 113-8656, Japan.
        {\tt\small kazuhiro@mist.i.u-tokyo.ac.jp}}
}
\begin{document}

\maketitle
\thispagestyle{empty}
\pagestyle{empty}

\begin{abstract}
Controllability scores provide control-theoretic centrality measures that quantify the relative importance of state nodes in networked dynamical systems. We establish a structural connection between finite-time controllability scoring and approximate optimal experimental design (OED): the finite-time controllability Gramian decomposes additively across nodes, yielding an affine matrix model of the same form as the information-matrix model in OED.
This yields a direct correspondence between the volumetric controllability score (VCS) and D-optimality, and between the average energy controllability score (AECS) and A-optimality, implying that the classical D/A invariance gap has a direct analogue in controllability scoring.
By contrast, we point out that controllability scoring generically admits a unique optimizer,
unlike approximate-OED formulations.
Finally, we uncover a long-horizon phenomenon with no OED counterpart: source-like state nodes without a negative self-loop can be increasingly downweighted by AECS as the horizon grows.
Two numerical examples corroborate this long-horizon downweighting behavior.
\end{abstract}


\section{Introduction}

Assessing the importance of individual state nodes in a dynamical network is a recurring need in control and network science \cite{bof2016role,liu2011controllability, pasqualetti2014controllability, summers2015submodularity}. Such assessments inform where to allocate actuation effort, how to prioritize interventions, and how to interpret large-scale models whose state variables may represent heterogeneous quantities. Yet, for many networked systems, it remains challenging to produce node-level importance scores that are both physically meaningful and robust to modeling conventions such as coordinate choices or unit rescalings.
Moreover, if a large-scale network is actuated through a single node, standard Gramian-based controllability metrics can deteriorate rapidly with the system size, making them numerically unreliable even when the system is controllable \cite{baggio2024controllability}. This issue affects node-ranking approaches that evaluate each node in isolation using a single actuation model, as in \cite[Section~III-G]{summers2015submodularity}. 

Motivated by this challenge, in our prior work~\cite{sato2022controllability}, we introduced a controllability scoring framework based on a virtual actuation viewpoint and proposed two node-allocation scores: the volumetric controllability score (VCS) and the average-energy controllability score (AECS). The core idea is to introduce a one-to-one correspondence between state nodes and virtual input channels, distribute a fixed actuation resource across these channels, and interpret the resulting optimal allocation over the probability simplex as a normalized node score.
Subsequent work established uniqueness properties and demonstrated applicability
to human brain networks~\cite{sato2025uniqueness}. The formulation was also
extended to accommodate constraints on actuated nodes~\cite{Sato2025}.

The present paper builds on this framework and addresses two conceptual and practical questions:
\begin{itemize}
    \item How does controllability scoring relate to classical approximate optimal experimental design (OED) \cite{harman2024polytope,huan2024optimal, pukelsheim2006optimal,jones2021optimal}, whose modern formulation dates back to the 1950s \cite{elfving1952optimum, kiefer1959optimum}?

    \item  How does the network structure encoded in the system matrix govern controllability scores over long time horizons?
\end{itemize}
The second question goes beyond the time-horizon dependence examined in
\cite[Section~III-C]{sato2025uniqueness}, where the role of the network structure was not explicitly analyzed.
Together, these questions clarify which aspects of controllability scoring can be explained through approximate OED and which aspects arise from the dynamics of networked systems.

\begin{table}[t]
\centering
\caption{Correspondence between controllability scoring and approximate OED, together with the resulting coordinate invariance properties.}
\label{tab:oed_cs_visual_simple}
\renewcommand{\arraystretch}{1.15}
\setlength{\tabcolsep}{5pt}
\begin{tabular}{@{}lll@{}}
\toprule
\textbf{OED} & \textbf{Controllability Score} & \textbf{Coordinate invariance} \\
\midrule
D-optimal design & VCS  & Invariant \\
A-optimal design & AECS & Generally dependent \\
\bottomrule
\end{tabular}
\end{table}

The contributions of this paper are summarized as follows:
\begin{enumerate}[i)]
    \item 
    We clarify the connection between controllability scoring and approximate OED.
    Under the virtual actuation model, the finite-horizon controllability Gramian admits an additive decomposition across state nodes, yielding an affine matrix model that is structurally identical to the information-matrix model in approximate OED.
    This establishes a direct correspondence (Table~\ref{tab:oed_cs_visual_simple}): VCS corresponds to D-optimality, whereas AECS corresponds to A-optimality.
    Motivated by this analogy, we further show an invariance gap under state-coordinate changes:
VCS is invariant under any nonsingular transformation, while AECS generally depends on the chosen coordinates (and thus can vary with unit changes).

    \item 
    We highlight a fundamental contrast between approximate OED and controllability scoring.
    Optimal allocations in approximate OED need not be unique~\cite{harman2024polytope},
    whereas controllability-scoring allocations are generically unique
~\cite{sato2022controllability,sato2025uniqueness}.
    This distinction is crucial because controllability scoring interprets the optimizer itself as centrality-like scores.
      Moreover, unlike the approximate OED framework, controllability scoring involves an intrinsic time-horizon parameter $T$ through the finite-time controllability Gramian, and the resulting $T\to\infty$ limit is a distinctive issue specific to controllability scores.

    \item 
    We provide a structural understanding of how the system matrix (i.e., the network coupling pattern) shapes the AECS allocation  over long horizons.
In particular, we identify source-like nodes, namely nodes that influence other nodes but receive little
or no incoming influence, without negative self-loops as a key architectural feature of the system matrix that can systematically drive AECS to downweight the corresponding nodes as the horizon grows. This mechanism is intrinsic to controllability scoring with no direct analogue in approximate OED. We corroborate this behavior in numerical experiments.

\end{enumerate}
Overall, the novelty of this paper lies in revealing the optimization structure underlying controllability scoring, rather than in proposing another controllability-based centrality index. 
This perspective clarifies the design-theoretic interpretation of VCS and AECS, the role of optimizer uniqueness, and the long-horizon influence of network dynamics on AECS allocations.

The remainder of the paper is organized as follows.
We review controllability scoring in Section~\ref{sec:Prelim}.
In Section~\ref{sec:oed}, we discuss the similarities and differences between controllability scoring and approximate OED.
Section~\ref{sec:source_node_mechanism} analyzes the long-horizon AECS
allocation and identifies the source-node mechanism.
Numerical examples are provided in Section~\ref{sec:examples}.
Finally, Section~\ref{Sec:conclusion} concludes the paper.

{\bf Notation:}
Let $\bb{N}$, $\bb{R}$, and $\bb{C}$ denote the sets of natural numbers, real numbers, and complex numbers, respectively.
We denote the imaginary unit by ${\rm i}$.
For a matrix $X$, $X^\top$ denotes the transpose.
For $X\in\mathbb C^{m\times n}$, the Hermitian transpose is defined by
$X^\dagger := (\overline{X})^{\top}$,
where $\overline{X}$ denotes the elementwise complex conjugate of $X$.
For a square matrix $X$, $\tr(X)$, $\det(X)$, and $\exp(X)$ denote the trace, determinant, and matrix exponential of $X$, respectively.
We write $X\succeq 0$ (resp., $X\succ 0$) to denote that $X$ is Hermitian positive semidefinite (resp., positive definite).
For any matrix $M\in\mathbb{C}^{n\times n}$ and index sets $I,J\subseteq\{1,\ldots,n\}$,
we denote by $M_{I,\,J}$ the submatrix with rows in $I$ and columns in $J$.
When $I=\{i\}$ and $J=\{j\}$, we simply write $M_{ij}$ instead of $M_{I, J}$;
when $I=\{i\}$ (resp., $J=\{j\}$), $M_{i,\,J}$ (resp., $M_{I,\,j}$) denotes the corresponding row (resp., column) submatrix.
For a vector $x$, $x_i$ denotes its $i$th component.
The identity matrix of size $n$ is denoted by $I_n$.
We also write $I$ for an identity matrix of appropriate size when there is no ambiguity.
The symbol $\bm{1}$ denotes all-ones vector in $\Real^n$, and $e_i$ is the $i$th standard basis vector.
The symbol $\diag(a_1,\ldots, a_n)$ denotes the diagonal matrix with diagonal entries $a_1,\dots,a_n$.
For $k\in\mathbb{N}$, let
    $\Delta_k := \left\{ x\in\Real^k\, \middle|\, \sum_{i=1}^k x_i = 1,\ x_i \ge 0 \right\}$.

\section{Controllability Scores} \label{sec:Prelim}

This section recalls the controllability-scoring framework \cite{sato2022controllability}.
We start from an autonomous linear-time invariant (LTI) network system
\begin{align}
    \dot x(t) = Ax(t), \label{eq:autonomous}
\end{align}
where $x(t)=\begin{pmatrix}
    x_1(t) &\cdots & x_n(t)
\end{pmatrix}^\top\in\Real^n$ and $x_i(t)$ denotes the state of node $i$.
The matrix $A\in\Real^{n\times n}$ encodes weighted interactions among nodes.
For autonomous dynamics \eqref{eq:autonomous}, the notion of controllability is not directly applicable because there is no physical input channel.
Therefore, following the virtual-actuation viewpoint, we introduce virtual input channels at all state nodes and use them to quantify
where to intervene and how strongly.

\subsection{Virtual actuation model and the corresponding controllability Gramian}

To define controllability scores, we augment \eqref{eq:autonomous} with virtual inputs and consider
the following virtually actuated system:
\begin{align}
  \dot x(t) = Ax(t) + B(p)u(t),\label{eq:lti}
\end{align}
where $B(p) := \diag(\sqrt{p_1},\ldots,\sqrt{p_n})$.
This setting establishes a one-to-one correspondence between each state node $x_i$ and a virtual input channel $u_i$.
For any horizon $T>0$, the finite-time controllability Gramian of \eqref{eq:lti} is defined as
\begin{align}
    W(p, T)
     &:= \int_0^T \exp(At) B(p)B(p)^\top \exp(A^\top t)\, \D t  \nonumber\\
     &= \sum_{i=1}^n p_i W_i(T), \label{Def_Wc}
\end{align}
where
    $W_i(T) := \int_0^T \exp(At) e_ie_i^\top \exp(A^\top t)\, \D t$.
Thus, $W_i(T)$ quantifies the controllability contribution of the single virtual channel acting on node $i$ over $[0,T]$,
and $W(p,T)$ is a weighted sum of $W_1(T),\ldots, W_n(T)$.

\subsection{Finite-time controllability scoring problem}

We define controllability scores as optimal allocations of the virtual actuation budget:
\begin{framed}
\vspace{-1em}
 \begin{align}
\label{prob:FTCSP}
    \begin{aligned}
        &&& \text{minimize} && h_T(p) \\
        &&& \text{subject to} && p \in X_T\cap \Delta_n.
    \end{aligned}
\end{align}
\vspace{-1em}
\end{framed}
\noindent
Here, $X_T$ is defined as
    $X_T := \{p\in \Real^n \mid W(p,T)\succ O\}$.
The constraint $p\in X_T$ ensures that virtual system \eqref{eq:lti} is controllable.
Moreover, $p \in \Delta_n$ encodes relative importance under a fixed budget.
The objective function $h_T(p)$ is chosen as either
\begin{align}
    f_T(p) &:= -\log\det W(p, T), \label{def_fT}\\
    g_T(p) &:= \tr \!\left(W(p, T)^{-1}\right). \label{def_gT}
\end{align}
When $h_T=f_T$, any optimizer of Problem~\eqref{prob:FTCSP} is called a volumetric controllability score (VCS);
when $h_T=g_T$, any optimizer of Problem~\eqref{prob:FTCSP} is called an average-energy controllability score (AECS).
Intuitively, VCS promotes enlarging the controllability ellipsoid,
whereas AECS promotes reducing the average control energy.
In both cases, the optimizer is interpreted as a node-level centrality score:
a larger $p_i$ assigns more virtual actuation authority (i.e., a larger share of the intervention budget) to node $i$, indicating that intervening at that node is expected to be more effective in influencing the network dynamics over $[0,T]$.

\subsection{Algorithm for solving Problem~\eqref{prob:FTCSP}}

Problem~\eqref{prob:FTCSP} can be solved  using a projected-gradient method with projections onto the simplex $\Delta_n$.
We refer to \cite{sato2022controllability, sato2025uniqueness} for algorithmic details and convergence guarantees.

\begin{remark}
        Unlike classical control problems, controllability scoring is motivated by networked systems (e.g., brain or social networks) where designing and applying continuous control inputs is often infeasible. Interventions are instead implemented through high-level actions such as policies, resource allocations, or localized treatments, rather than engineered actuators. Accordingly, the goal of controllability scoring is not controller synthesis but intervention planning: identifying where limited intervention capability should be placed to most effectively influence the network dynamics.
\end{remark}

\section{Relation Between Problem~\eqref{prob:FTCSP} and Optimal Experimental Design} \label{sec:oed}

This section clarifies the relationship between controllability scoring
and classical optimal experimental design (OED) \cite{pukelsheim2006optimal}.

\subsection{Approximate OED} \label{subsec_OED}

We briefly recall the classical setting of approximate OED in linear regression; see, e.g., \cite{harman2024polytope}.
Consider the model
\begin{align}
    y = F\theta + \varepsilon,\qquad
    y\in\Real^N,\ \theta\in\Real^d,
\end{align}
where the $i$th row of $F$ is the regression vector $f(x_i)^\top$
associated with an experimental condition $x_i$
and $\varepsilon$ is a zero-mean Gaussian noise vector with covariance $\sigma^2 I_N$.
Suppose that there are $m$ distinct conditions $x_1,\dots,x_m$
and that condition $x_i$ is used $k_i$ times, so that
$N = \sum_{i=1}^m k_i$ and $w_i := k_i/N$ defines the empirical design
weights $w\in\Delta_m$.

The least-squares estimator $\hat\theta = (F^\top F)^{-1}F^\top y$ is
unbiased with covariance
    $\sigma^2 (F^\top F)^{-1}$.
The Fisher information matrix is therefore
\[
    M(w)
    := \frac{1}{\sigma^2}F^\top F
    = \frac{N}{\sigma^2}\sum_{i=1}^m w_i f(x_i)f(x_i)^\top
    = \sum_{i=1}^m w_i M_i,
\]
where $M_i := \frac{N}{\sigma^2} f(x_i)f(x_i)^\top$ are the elementary
information matrices.

Based on the information matrix $M(w)$,
approximate OED problems can be formulated as
\begin{framed}
\vspace{-1.3em}
 \begin{align}
\label{prob:OED}
    \begin{aligned}
        &&& \text{minimize} && \gamma(w) \\
        &&& \text{subject to} && w \in \Delta_m.
    \end{aligned}
\end{align}
\vspace{-1.3em}
\end{framed}

\noindent
Here, the objective function $\gamma(w)$ is chosen as either
\begin{align}
    \alpha(w) := -\log\det M(w), \quad
    \beta(w) := \tr \!\left(M(w)^{-1}\right).
\end{align}
When $\gamma=\alpha$, any minimizer $w^*$ is called a D-optimal design;
when $\gamma=\beta$, any minimizer $w^*$ is called an A-optimal design.

Problem~\eqref{prob:OED} is directly comparable to Problem~\eqref{prob:FTCSP}; see Table~\ref{tab:oed_cs_correspondence} for the correspondence.
Unlike Problem~\eqref{prob:FTCSP}, where $p\in X_T$ explicitly enforces $W(p,T)\succ O$, Problem~\eqref{prob:OED} is often written only with the simplex constraint $w\in\Delta_m$.
This correspondence is structural in the sense that both problems
optimize the same convex criteria over an affine matrix model. However,
the meanings of the weights are different: \(p_i\) allocates virtual
actuation authority to state node \(i\) and is interpreted as a
controllability-based centrality score, whereas \(w_i\) specifies the
relative frequency of using the experimental condition \(x_i\) and is
interpreted as an experimental design. Hence \(W_i(T)\) and \(M_i\)
are algebraically analogous but semantically different objects.

\begin{table}[t]
\centering
\caption{Structural correspondence between Problem~\eqref{prob:FTCSP} and Problem~\eqref{prob:OED}.}
\label{tab:oed_cs_correspondence}
\renewcommand{\arraystretch}{1.15}
\setlength{\tabcolsep}{5pt}
\begin{tabular}{@{}ll@{}}
\toprule
Problem~\eqref{prob:FTCSP} & Problem~\eqref{prob:OED} \\
\midrule
$p$ & $w$ \\
$W_i(T)$ & $M_i$  \\
$f_T(p)$ & $\alpha(w)$ \\
$g_T(p)$ & $\beta(w)$ \\
\bottomrule
\end{tabular}
\end{table}

\subsection{Invariance of VCS and coordinate dependence of AECS} \label{sec:scaling}

It is well known in OED \cite{pukelsheim2006optimal} that D-optimality is invariant under nonsingular reparameterizations, whereas A-optimality generally depends on the chosen coordinates. As summarized in Table~\ref{tab:oed_cs_visual_simple}, we show that the same phenomenon arises for controllability scoring under state-coordinate changes: VCS is invariant, while AECS is not in general.

Consider the state-coordinate change (including rescalings)
\begin{align}
    x = S\tilde x, \label{eq:state_scaling}
\end{align}
where $S\in {\bb C}^{n\times n}$ is any nonsingular (possibly complex-valued) matrix.
Then,
virtual system \eqref{eq:lti} is rewritten as
    $\dot{\tilde x}(t)
    = \tilde A \tilde x(t) + \tilde B(p)u(t)$, 
where $\tilde A := S^{-1}AS$ and $\tilde B(p):=S^{-1}B(p)$.
The corresponding finite-time controllability Gramian satisfies the standard
congruence relation
\begin{align}
    \widetilde W(p,T) = S^{-1}W(p,T)S^{-\dagger}. \label{eq:Gramian_congruence}
\end{align}
This implies that
    $\widetilde W(p,T)\succ O \quad \Leftrightarrow\quad W(p,T)\succ O$.
Thus, the feasibility of Problem~\eqref{prob:FTCSP} is unaffected by coordinate-change \eqref{eq:state_scaling}.

Accordingly, the counterpart of Problem~\eqref{prob:FTCSP} under \eqref{eq:state_scaling} is formulated as

\begin{framed}
\vspace{-1.3em}
 \begin{align}
\label{prob:FTCSP2}
    \begin{aligned}
        &&& \text{minimize} && \widetilde{h}_T(p) \\
        &&& \text{subject to} && p \in X_T\cap \Delta_n.
    \end{aligned}
\end{align}
\vspace{-1.3em}
\end{framed}
\noindent
The objective function $\widetilde{h}_T(p)$ is chosen as either
\begin{align}
    \widetilde{f}_T(p) &:= -\log\det \widetilde{W}(p, T), \label{def_fT2}\\
    \widetilde{g}_T(p) &:= \tr \!\left( \widetilde{W}(p, T)^{-1}\right). \label{def_gT2}
\end{align}

The following proposition shows that VCS is invariant
under nonsingular state-coordinate transformations, familiar from D-optimal design.
Consequently, the VCS-induced node scoring is robust
against coordinate changes.

\begin{proposition}[VCS invariance] \label{prop:VCS_invariance}
Let $A\in\mathbb{R}^{n\times n}$ be arbitrary and let $S\in\mathbb{C}^{n\times n}$
be any nonsingular matrix.
Then, for almost all $T>0$, the optimal solutions
$p_{\rm VCS}$ of Problem~\eqref{prob:FTCSP} with $h_T=f_T$
and $\widetilde p_{\rm VCS}$ of Problem~\eqref{prob:FTCSP2}
with $\widetilde h_T=\widetilde f_T$ are unique and satisfy
\begin{align}
    \widetilde p_{\rm VCS} = p_{\rm VCS}. \label{eq:VCS_invariance_statement}
\end{align}
\end{proposition}

\begin{proof}
It follows from \eqref{eq:Gramian_congruence} and \eqref{def_fT2} that
    $\widetilde{f}_T(p) 
    = -\log\det \left(S^{-1}W(p,T)S^{-\dagger}\right)
    = f_T(p) + 2\log|\det S|$.
Since the second term is independent of \(p\), the two problems have the same
optimal solution set.  Moreover, by \cite[Theorem~1]{sato2025uniqueness}, the
optimal solution is unique for almost all \(T>0\).  Hence, \eqref{eq:VCS_invariance_statement} holds for almost all $T>0$.
\end{proof}

In contrast, the following example shows that
AECS can be sensitive to a state-coordinate change.
In heterogeneous networks, the sensitivity is not a minor technicality:
it can alter node rankings even when the underlying dynamics are unchanged.

\begin{example}
Let $n=2$, $A=0$, and fix any $T>0$.
Then 
    $W(p,T)=\int_0^T B(p)B(p)^\top \D t
    =T{\rm diag}(p_1,p_2)$.
Thus, the constraint $p\in X_T$ yields
    ${\rm tr} \left(W(p,T)^{-1}\right)
    =\frac{1}{T}\left(\frac{1}{p_1}+\frac{1}{p_2}\right)$.
The Lagrange multiplier method implies that 
\begin{align}
(p_1, p_2) = (1/2, 1/2) \label{AECS_ex}    
\end{align}
is the optimal solution, i.e., AECS.

Next, consider the diagonal scaling, i.e., a state-coordinate change, $x=S\tilde x$ with
$S=\diag(s_1,s_2)$ where $s_1>0$ and $s_2>0$.
Then, \eqref{eq:Gramian_congruence} yields
    $\widetilde W(p,T)
    =T\,\diag\!\left(\frac{p_1}{s_1^2},\frac{p_2}{s_2^2}\right)$,
and thus if $p\in X_T$, we have
${\rm tr} \left(\widetilde W(p,T)^{-1}\right)
    =\frac{1}{T}\left(\frac{s_1^2}{p_1}+\frac{s_2^2}{p_2}\right)$,
and the Lagrange multiplier method implies that
    $(p_1, p_2) = \left(\frac{s_1}{s_1+s_2}, \frac{s_2}{s_1+s_2}\right)$,
which differs from \eqref{AECS_ex} whenever $s_1\neq s_2$.
Thus, the AECS optimizer is not invariant under nonsingular
state-coordinate changes in general. 
\end{example}

\begin{remark}[On transpose vs.\ Hermitian transpose]
\label{rem:Hermitian_Gramian}
In the real-valued setting, the finite-time controllability Gramian $W(p,T)$ is defined as
\eqref{Def_Wc}.
However, if we allow complex-valued state-coordinate change \eqref{eq:state_scaling}, then \eqref{Def_Wc} must be interpreted
with the transpose replaced by the Hermitian transpose:
\begin{align}
    W(p,T)=\int_{0}^{T} \exp(At)B(p)B^{\dagger}(p)\exp(A^{\dagger}t)\, \D t,
\end{align}
so that $W(p,T)$ is Hermitian and positive semidefinite, consistent with the
standard energy inner product $\langle z,w\rangle=z^{\dagger}w$ on
$\mathbb C^{n}$.
This distinction is essential: keeping $^{\top}$ while allowing complex $S$
can destroy positive definiteness under a coordinate change.
For instance, for any $W\succ 0$ and $S={\rm i}I$,
$S^{-1}WS^{-\top}=(-{\rm i})W(-{\rm i})=-W \prec 0$,
whereas
$S^{-1}WS^{-\dagger}=(-{\rm i})W({\rm i})=W\succ 0$.
\end{remark}

\begin{remark}
Proposition~\ref{prop:VCS_invariance} concerns invariance with respect to
a change of state representation.  The weight \(p_i\) is interpreted as
the weight assigned to the \(i\)th original virtual input channel, or to
the \(i\)th original state node before the coordinate transformation.
For a general nonsingular matrix \(S\), the transformed coordinates
\(\tilde x_i\) may be linear combinations of the original states and
therefore need not define the same nodes.  Hence, the proposition does
not assert invariance of scores for newly defined nodes in the transformed
coordinates; it asserts invariance of the VCS weights assigned to the
original input channels under different state representations.

By contrast, AECS can change even under diagonal state rescalings, where
the node correspondence is preserved.  Thus, AECS may depend on the
chosen state coordinates even when the underlying dynamics are unchanged.
\end{remark}

\subsection{Key differences from approximate OED}

Although Problem~\eqref{prob:FTCSP} and Problem~\eqref{prob:OED} are identical in optimization form up to the extra feasibility constraint $p\in X_T$ in Problem~\eqref{prob:FTCSP}, they differ in several important respects.
These differences are crucial because the optimizer in controllability scoring
is interpreted as a node-level centrality score, rather than merely as an
optimal allocation of experimental effort.

\subsubsection{Consequences of the semantic difference} \label{sec_different_semantics}

The semantic difference described after Table~\ref{tab:oed_cs_correspondence}
has concrete implications. In controllability scoring, \(W_i(T)\) is
generated by a virtual input at state node \(i\) and its propagation
through \(\exp(At)\). Therefore, the optimizer of
Problem~\eqref{prob:FTCSP} depends on both the system matrix \(A\) and
the horizon \(T\), and can reflect structural properties of the network.

By contrast, in approximate OED, \(M_i\) represents the information
provided by candidate experimental condition \(i\), and the optimizer of
Problem~\eqref{prob:OED} specifies the relative usage frequencies of
these conditions for parameter estimation. Hence, the analogy is
structural and does not imply that the two optimizers have the same
interpretation.

\subsubsection{Uniqueness}

In Problem~\eqref{prob:OED}, the parameter dimension $d$ is often moderate, while
the number of candidate design points $m$ can be large. Each $M_i$ is a
symmetric $d\times d$ matrix and therefore lies in the $d(d+1)/2$--dimensional
vector space of symmetric matrices. It is thus common for $\{M_i\}_{i=1}^m$ to
be linearly dependent. Consequently, the set of optimal approximate designs is
often not a single point but a nontrivial polytope of optimal weights, as
analyzed in \cite{harman2024polytope}. In particular, D- or A-optimal designs
are typically not unique.

In controllability scoring, we instead aggregate the finite-time
controllability contributions $W_i(T)$ associated with the
virtual inputs at the state nodes. Here, the number of design variables equals
the state dimension $n$, and the
matrices $\{W_i(T)\}_{i=1}^n$ are generically linearly independent. Consequently, the optimal allocations defining VCS and AECS are typically unique \cite[Theorem~1]{sato2025uniqueness}, which is crucial for interpreting them as centrality-like scores for the state nodes.

\subsubsection{Intrinsic time horizon $T$ and the limit $T\to\infty$}

Unlike Problem~\eqref{prob:OED}, controllability scoring problem~\eqref{prob:FTCSP} involves an intrinsic time-horizon parameter $T>0$ through finite-time controllability Gramian $W(p,T)$ defined by \eqref{Def_Wc}.
The limit $T\to\infty$ is a new issue in controllability scoring,
as it removes dependence on a user-chosen terminal time and thereby improves reproducibility.
However, this limit cannot be addressed by the standard Gramian-based approach when $A$ is non-Hurwitz:
the controllability Gramian $W(p,T)$ diverges as $T\to\infty$, and even for large finite $T$ it can become severely ill-conditioned,
making the objective functions numerically unstable and rendering many Gramian-based centralities ill-defined on an infinite horizon.

Recent work~\cite{umezu2026controllability} resolves this obstacle by introducing a scaled controllability Gramian.
Formally, the scaled Gramian is obtained from
the finite-horizon Gramian by the congruence transformation
$\widetilde W(p,T):=D(T)^{-1}Q^{-1}W(p, T)Q^{-\dagger}D(T)^{-\dagger}$,
where \(Q\) brings \(A\) into its Jordan normal form and \(D(T)\) is a block-diagonal,
time-dependent scaling.  Thus, it has the same algebraic form as a
state-coordinate transformation with \(S=QD(T)\), although the scaling depends
on the terminal time \(T\).

For VCS, the congruence transformation changes the log-determinant objective
only by a \(p\)-independent term, so the scaled finite-horizon problem is
equivalent to the original one.  For AECS, the equivalent scaled formulation
uses
 ${\rm tr}\left(
    Q^{-\dagger}D(T)^{-\dagger}\widetilde W(p,T)^{-1}D(T)^{-1}Q^{-1}
    \right)$,
which equals the original objective \({\rm tr}(W(p,T)^{-1})\), rather than
\({\rm tr}(\widetilde W(p,T)^{-1})\).  Thus, the scaled AECS formulation preserves the original finite-horizon objective and, at the same time, makes explicit the asymptotic structure that governs its behavior as
\(T\to\infty\).

With these formulations, the scaled VCS problem is equivalent to the original
finite-horizon VCS problem up to an additive constant independent of $p$, and
the scaled AECS formulation is exactly equivalent to the original finite-horizon
AECS problem. Moreover, the scaled Gramian admits a well-defined limit as
$T\to\infty$, which makes it possible to define infinite-horizon VCS and AECS
by taking suitable limits of the corresponding scaled formulations.


\section{Why AECS downweights source-like nodes on long horizons}
\label{sec:source_node_mechanism}

This section provides a structural explanation for why AECS can yield node rankings that differ qualitatively from those induced by VCS on long horizons, as reported in \cite[Section~III-C-3]{sato2025uniqueness}.
The key mechanism is the presence of source-like state nodes, i.e.,
coordinates that are weakly affected by the rest of the network.
For such nodes, a diagonal entry of the finite-time controllability Gramian $W(p,T)$ defined by \eqref{Def_Wc}
admits an explicit form, leading to sharp lower bounds for AECS and revealing
a characteristic contrast to the VCS behavior on long horizons.
Importantly, this mechanism is specific to controllability scoring (Problem~\ref{prob:FTCSP}) and has no direct counterpart in classical OED (Problem~\ref{prob:OED}), as mentioned in Section~III-C-1.

\subsection{Explicit formula for the diagonal Gramian entry under a left-eigenvector condition}
\label{subsec:source_explicit_Wii}

We start from a structural condition that captures ``source-like" behavior
at the level of the system matrix $A$.

\begin{assumption}[Left-eigenvector condition]\label{ass:left_eig_source}
Fix an index $i\in\{1,\dots,n\}$.
Assume that there exists $\ell_i\in\mathbb{R}$ such that
\begin{equation}\label{eq:left_eig_source_assump}
    e_i^\top A = \ell_i\, e_i^\top,
\end{equation}
i.e., $e_i$ is a left eigenvector of $A$.
\end{assumption}

Assumption~\ref{ass:left_eig_source} means that the $i$th row of $A$ has no off-diagonal entries:
$A_{ij}=0$ for all $j\neq i$.
In particular, the scalar $\ell_i$ in \eqref{eq:left_eig_source_assump} is simply the diagonal entry,
$\ell_i = A_{ii}$.
Thus, $x_i$ is not driven by other state nodes.
In this sense, $x_i$ behaves as a source-like node, and the sign of $\ell_i$ specifies whether it has a nonnegative self-loop weight.

\begin{lemma} \label{Lem_important}
Suppose that Assumption~\ref{ass:left_eig_source} holds. Then, the $(i,i)$ entry of $W(p,T)$ is given by
\begin{equation}\label{eq:Wii_general}
    (W(p,T))_{ii}
    = p_i\omega_i(T),
\end{equation}
where 
\begin{align}
        \omega_i(T)=
    \begin{cases}
        T & {\rm if}\quad\ell_i=0,\\[2mm]
        \dfrac{\e^{2 \ell_iT}-1}{2 \ell_i} &{\rm if}\quad \ell_i\neq 0.
    \end{cases} \label{def_omega}
\end{align}
Moreover, if $p\in X_T$, then $p_i>0$.
\end{lemma}
\begin{proof}
By definition \eqref{Def_Wc},
    \begin{align}
        (W(p,T))_{ii}
    = e_i^\top W(p,T)e_i 
    = p_i \int_0^T \e^{2 \ell_i t}\, \D t.
    \end{align}
The second equality follows from Assumption~\ref{ass:left_eig_source}, which implies $e_i^\top \exp (At) = \e^{\ell_i t}\, e_i^\top$. 
Thus, \eqref{eq:Wii_general} holds. 
Moreover, if $p\in X_T$, then $W(p,T)\succ O$ and hence $(W(p,T))_{ii}>0$. Since $\omega_i(T)>0$, it follows from \eqref{eq:Wii_general} that $p_i>0$.
\end{proof}

\subsection{AECS lower bound}
\label{subsec:source_AECS_bound}

We provide a quantitative lower bound for the AECS objective function
$g_T(p)$ in Problem~\eqref{prob:FTCSP}.

\begin{theorem}\label{prop:AECS_lower_source}
Suppose that $p\in X_T$ and Assumption~\ref{ass:left_eig_source} hold. Then,
\begin{align}
    g_T(p) \geq \frac{1}{p_i\,\omega_i(T)},\label{eq:AECS_lower_source}
\end{align}
where $\omega_i$ is defined by \eqref{def_omega}.
\end{theorem}

\begin{proof}
By definition \eqref{def_gT} of $g_T$,
    $g_T(p) \geq \left(W(p,T)^{-1}\right)_{ii} 
    \geq \frac{1}{\left(W(p,T)\right)_{ii}}$.
The first inequality follows from the definition of the trace.
The second inequality follows from Lemma~\ref{lem:inv_diag_bound} in Appendix~\ref{Ape}.
Thus, Lemma~\ref{Lem_important} yields \eqref{eq:AECS_lower_source}.
\end{proof}

A particularly revealing case is $\ell_i=0$, for which $\omega_i(T)=T$ in
\eqref{eq:Wii_general}. Then, \eqref{eq:AECS_lower_source} becomes
\begin{equation}\label{eq:AECS_lower_1_over_pT}
    g_T(p)\ \ge\ \frac{1}{p_i\,T}.
\end{equation}
Thus, for any fixed $p_i>0$, this lower bound decays as $T^{-1}$ as $T\to\infty$.
Equivalently, the lower bound becomes less restrictive on long horizons,
which helps explain why the AECS optimizer may allocate a very small weight to
a source-like node when $T$ is large.

In contrast, if $\ell_i<0$,
\eqref{eq:AECS_lower_source} yields the non-decaying bound
\begin{align}
    g_T(p)\ \geq \frac{-2\ell_i}{p_i\,(1-{\rm e}^{2\ell_i T})} \to \frac{-2\ell_i}{p_i}\quad (T\to\infty). \label{nondecay}
\end{align}
Hence, unlike the marginal case $\ell_i=0$, strictly stable self-dynamics remove the horizon-induced relaxation of the lower bound and thus discourage the AECS optimizer from taking $p_i$ extremely small.
In fact, \eqref{nondecay} shows that making $p_i$ very small forces $g_T(p)$ to be large, so such allocations are penalized by the AECS objective.

\begin{remark}
    If $\ell_i>0$,
\eqref{eq:AECS_lower_source} yields
    $g_T(p) \geq \frac{2\ell_i}{p_i\,({\rm e}^{2\ell_i T}-1)}$.
As $T\to\infty$, the right-hand side decays exponentially.
Thus, the lower bound becomes even less restrictive than in the marginal case
$\ell_i=0$.
Accordingly, this bound alone does not prevent the AECS optimizer from taking
$p_i$ very small on long horizons.
\end{remark}

\begin{remark}[Nearly source-like nodes]
Assumption~\ref{ass:left_eig_source} is imposed to obtain the exact identity
\eqref{eq:Wii_general}. If the \(i\)th row of \(A\) has small
off-diagonal entries, this identity no longer holds exactly. Nevertheless,
for every fixed \(T\), the Gramian \(W(p,T)\) depends continuously on
\(A\), so small incoming couplings lead to small perturbations of the
AECS objective on finite horizons. Thus, the downweighting mechanism is
expected to persist qualitatively for nearly source-like nodes, although
we do not claim a uniform perturbation result as \(T\to\infty\).
\end{remark}

\begin{remark}[Geometric and energy interpretation]
The AECS downweighting effect can be interpreted through the eigenvalues
of the controllability Gramian. Let
\(0<\lambda_1(p,T)\leq \cdots \leq \lambda_n(p,T)\) be the eigenvalues
of \(W(p,T)\). Since
$g_T(p)= \sum_{k=1}^n \frac{1}{\lambda_k(p,T)}$,
small eigenvalues contribute disproportionately through their reciprocals.
Thus,
decreasing \(g_T(p)\) mainly requires enlarging the small eigenvalues of
\(W(p,T)\), which correspond to directions in which the reachable ellipsoid has small
semi-axis lengths and hence require large control energy. 

This suggests a
possible mechanism behind AECS downweighting: for a source-like node with
weak self-dynamics, its Gramian contribution can become large over a long
horizon, but if it mainly enlarges already well-reachable directions,
additional weight on that node may yield only a small marginal decrease
in \(g_T(p)\). Other nodes can be more effective when their contributions
enlarge poorly reachable directions, corresponding to smaller eigenvalues
of \(W(p,T)\).
\end{remark}

\subsection{VCS lower bound}
\label{subsec:det_decomp_exact_source}

Unlike the AECS objective function $g_T$,
the VCS objective $f_T$ does not reduce to a bound in terms of
$p_i\omega_i(T)$ alone:

\begin{theorem}
Suppose that $p\in X_T$ and Assumption~\ref{ass:left_eig_source} hold. Then,
\begin{align}
    f_T(p)
    \geq -\log (p_i\omega_i(T))
    -\log\det \left(W(p,T)\right)_{\bar i,\,\bar i}, \label{eq:f_decomp_exact}
\end{align}    
where  $\bar i := \{1,\dots,n\}\setminus\{i\}$ and
$\omega_i$ is defined in \eqref{def_omega}.
\end{theorem}
\begin{proof}
Since $p\in X_T$, $W(p,T)\succ 0$.
Thus, Lemma~\ref{lem:VCS} in Appendix~\ref{Ape} implies that
\begin{equation}\label{eq:det_schur_exact}
    \det W(p,T)
    \leq (W(p,T))_{ii} \det\left(\left(W(p,T)\right)_{\bar i,\,\bar i}\right).
\end{equation}
Thus, Lemma~\ref{Lem_important} yields \eqref{eq:f_decomp_exact}.
\end{proof}

The bound~\eqref{eq:f_decomp_exact} makes explicit that the lower bound of
$f_T$ depends on $p_i$ not only through $-\log(p_i\omega_i(T))$ but also through the term
$-\log\det \left(W(p,T)\right)_{\bar i,\,\bar i}$,
which depends on the overall allocation $p$.
Consequently, Assumption~\ref{ass:left_eig_source} alone does not justify the conclusion that
choosing $p_i$ small is beneficial for VCS on long horizons.


\section{Examples}\label{sec:examples}

This section illustrates the qualitative discrepancy between VCS and AECS by two examples.
We first present a minimal two-node model, for which the finite-time Gramian admits a closed-form expression.
We then revisit the $n=10$ directed network used in our earlier numerical study~\cite{sato2025uniqueness} to demonstrate that the same phenomenon appears in a larger network and to examine how it changes when a negative self-loop is added to a source-like node.

\subsection{A closed-form two-node diagonal example}\label{subsec:ex_two_node_diag}

We consider the $2\times 2$ diagonal dynamics \eqref{eq:lti} with
   $A=\begin{pmatrix}
        0 & 0\\
        0 & -1
    \end{pmatrix}$.
Both nodes are source-like in the sense of Assumption~\ref{ass:left_eig_source},
since each row of $A$ has no off-diagonal entries and hence
$e_1^\top A = 0\cdot e_1^\top$, $e_2^\top A = (-1)\cdot e_2^\top$. In particular, node~1 is marginal ($\ell_1=0$), whereas node~2 is strictly stable ($\ell_2=-1$).
For this system, the finite-time controllability Gramian $W(p,T)$ defined in \eqref{Def_Wc} is given by 
    $W(p,T) =
    \begin{bmatrix}
        p_1T & 0\\
        0 & p_2\,\zeta(T)
    \end{bmatrix}$,
    $\zeta(T) :=\int_0^T e^{-2t}\, \D t=\frac{1-e^{-2T}}{2}$.
Note that $p=(p_1,p_2)\in X_T\cap\Delta_2$ implies that $0<p_1, p_2 < 1$.

\subsubsection{VCS}
Since the matrix $A$ is symmetric, \cite[Theorem~2]{sato2025uniqueness} implies that
the unique VCS optimizer is
    $(p_1^{\rm VCS}(T), p_2^{\rm VCS}(T))=(1/2, 1/2)$
for all $T>0$.

\subsubsection{AECS}
By \eqref{def_gT}, for $p=(p_1,p_2)\in X_T\cap\Delta_2$, we have
    $g_T(p)=\frac{1}{Tp_1}+\frac{1}{p_2\zeta(T)}$.
Thus, the Lagrange multiplier method yields the unique AECS optimizer
\begin{align}
    p_{1}^{\rm AECS}(T) &= \frac{1}{1+\sqrt{\dfrac{2T}{1-\e^{-2T}}}},\\ 
    p_2^{\rm AECS}(T) &= 1-p_1^{\rm AECS}(T).
\end{align}
In particular, $p_{1}^{\rm AECS}(T)\to 0$ as $T\to\infty$.
This illustrates, in the simplest closed form, how a source-like node with a nonnegative self-loop can be
assigned vanishing weight in the AECS objective function $g_T$ on long horizons,
in contrast to the VCS objective function $f_T$.

\subsection{$10$-node directed network examples}\label{subsec:ex_ten_node}

We next consider the directed network in Fig.~\ref{fig:network}, originally studied in~\cite[Section~III-C-3]{sato2025uniqueness}, to illustrate how the long-horizon behavior of the scores depends on the presence of self-loops at source-like nodes.

\subsubsection{Source-like nodes without self-loops}

The network in Fig.~\ref{fig:network} consists of $n=10$ nodes and all edges have a uniform weight
$c=0.2$.
Let $L$ denote the (directed) graph Laplacian associated with this network and
set $A=-L$.

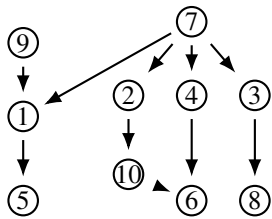
\begin{figure}[htbp]
    \centering
    \begin{tikzpicture}[scale=0.9]
        \node (n9) at (0.0, 3.2) {9};
        \node (n1) at (0.0, 1.8) {1};
        \node (n5) at (0.0, 0.2) {5};

        \node (n7)  at (3.2, 3.6) {7};
        \node (n2)  at (2.0, 2.2) {2};
        \node (n4)  at (3.2, 2.2) {4};
        \node (n3)  at (4.4, 2.2) {3};

        \node (n10) at (2.0, 0.7) {10};
        \node (n6)  at (3.2, 0.2) {6};
        \node (n8)  at (4.4, 0.2) {8};

        \foreach \v in {n1,n2,n3,n4,n5,n6,n7,n8,n9,n10}
            \draw[thick] (\v) circle[radius=0.28];

        \draw[-Latex,thick,shorten >=2pt,shorten <=2pt] (n1)--(n5);
        \draw[-Latex,thick,shorten >=2pt,shorten <=2pt] (n2)--(n10);
        \draw[-Latex,thick,shorten >=2pt,shorten <=2pt] (n3)--(n8);
        \draw[-Latex,thick,shorten >=2pt,shorten <=2pt] (n4)--(n6);

        \draw[-Latex,thick,shorten >=2pt,shorten <=2pt] (n7)--(n1);
        \draw[-Latex,thick,shorten >=2pt,shorten <=2pt] (n7)--(n2);
        \draw[-Latex,thick,shorten >=2pt,shorten <=2pt] (n7)--(n3);
        \draw[-Latex,thick,shorten >=2pt,shorten <=2pt] (n7)--(n4);

        \draw[-Latex,thick,shorten >=2pt,shorten <=2pt] (n9)--(n1);
        \draw[-Latex,thick,shorten >=2pt,shorten <=2pt] (n10)--(n6);
    \end{tikzpicture}
    \caption{Network employed in the numerical experiments.}
    \label{fig:network}
\end{figure}

This network contains two source-like nodes (nodes~7 and~9) in the sense that they have no incoming edges, and hence their corresponding rows of $A=-L$ have no off-diagonal entries.
In particular, both nodes satisfy Assumption~\ref{ass:left_eig_source} with $\ell_i=A_{ii}=0$.
The marginal case $\ell_i=0$ yields the lower bound
\eqref{eq:AECS_lower_1_over_pT}.

Table~\ref{tab:combined} shows that this mechanism manifests sharply for node~9:
its AECS allocation drops from $9.1\times 10^{-2}$ at $T=1$ to $7.0\times 10^{-3}$ at $T=1000$ and further to $2.3\times 10^{-3}$ at $T=10000$.
By contrast, the AECS allocation of node~7 remains bounded away from zero even for $T=10000$.
This disparity is consistent with the network topology.
Although both nodes are source-like in the sense of having no incoming edges, node~7 has multiple outgoing edges (it directly influences nodes~1--4), so allocating actuation budget to node~7 improves the reachability of several downstream states and continues to reduce the average control energy captured by $g_T(p)$.
Node~9, in contrast, has a more limited downstream reach in this topology (it primarily feeds into node~1 and then node~5).
Since the virtual-actuation model allows direct allocation to other nodes, comparable downstream improvements can often be achieved by allocating budget elsewhere, which is consistent with the strong downweighting of node~9 under AECS for large $T$.

VCS exhibits the opposite behavior on long horizons.
While the allocations are nearly uniform for very short horizons ($T=0.01$), VCS increasingly concentrates weight on upstream (source-like) nodes as $T$ increases, assigning substantial mass to nodes~7 and~9 (approximately $0.249$ and $0.166$, respectively, at $T=10000$).
This contrast can be traced back to the distinct analytic structures of the two objectives.
In fact, the VCS objective function $f_T$ does not admit an analogous reduction to a bound depending only on $p_i\omega_i(T)$ as in \eqref{eq:AECS_lower_1_over_pT}:
Under Assumption~\ref{ass:left_eig_source}, \eqref{eq:f_decomp_exact} shows that the influence of $p_i$ on $f_T$ remains coupled to the overall allocation through the principal submatrix determinant term.
As a result, even when a node is source-like, VCS need not favor driving its allocation toward zero on long horizons, which is consistent with the non-negligible VCS weights observed for nodes~7 and~9 in Table~\ref{tab:combined}.

\begin{table*}[t]
\centering
\caption{Evaluation of VCS and AECS for various $T$ reported in \cite[Table III]{sato2025uniqueness}.}
\label{tab:combined}
\begin{minipage}[b]{0.48\textwidth}
\centering
\textbf{VCS}\\[0.5ex]
\begin{tabular}{c c c c c}
\toprule
\textbf{Node} & \textbf{$T=0.01$} & \textbf{$T=1$} & \textbf{$T=1000$} & \textbf{$T=10000$} \\
\midrule
1  & 0.1000 & 0.0997 & 0.0733 & 0.0733 \\
2  & 0.1000 & 0.1000 & 0.1011 & 0.1011 \\
3  & 0.1000 & 0.1000 & 0.1088 & 0.1087 \\
4  & 0.1000 & 0.1000 & 0.0864 & 0.0864 \\
5  & 0.1000 & 0.0997 & 0.0456 & 0.0450 \\
6  & 0.1000 & 0.0994 & 0.0607 & 0.0607 \\
7  & 0.1000 & 0.1013 & 0.2493 & 0.2495 \\
8  & 0.1000 & 0.0997 & 0.0423 & 0.0422 \\
9  & 0.1000 & 0.1003 & 0.1661 & 0.1667 \\
10 & 0.1000 & 0.1000 & 0.0664 & 0.0663 \\
\bottomrule
\end{tabular}
\end{minipage}%
\hfill
\begin{minipage}[b]{0.48\textwidth}
\centering
\textbf{AECS}\\[0.5ex]
\begin{tabular}{c c c c c}
\toprule
\textbf{Node} & \textbf{$T=0.01$} & \textbf{$T=1$} & \textbf{$T=1000$} & \textbf{$T=10000$} \\
\midrule
1  & 0.1000 & 0.1093 & 0.1713 & 0.1728 \\
2  & 0.1000 & 0.1000 & 0.1133 & 0.1136 \\
3  & 0.1000 & 0.1000 & 0.1205 & 0.1209 \\
4  & 0.1000 & 0.1000 & 0.1058 & 0.1061 \\
5  & 0.1000 & 0.0998 & 0.0907 & 0.0923 \\
6  & 0.1000 & 0.1091 & 0.1335 & 0.1338 \\
7  & 0.1000 & 0.0913 & 0.0926 & 0.0928 \\
8  & 0.1000 & 0.0998 & 0.0695 & 0.0694 \\
9  & 0.1000 & 0.0908 & 0.0070 & 0.0023 \\
10 & 0.1000 & 0.1000 & 0.0957 & 0.0959 \\
\bottomrule
\end{tabular}
\end{minipage}
\end{table*}

\subsubsection{Source-like nodes with negative self-loops}

We next consider a self-loop variant of the network in Fig.~\ref{fig:network}.
The underlying directed topology and all off-diagonal edge weights remain unchanged; the only modification is the addition of a negative self-loop at node~9:
We replaced the $(9,9)$ entry by $A_{99}=-1$ while keeping $A_{ij}$ unchanged for all $(i,j)\neq(9,9)$.

\begin{table*}[t]
\centering
\caption{Evaluation of VCS and AECS for various $T$ on the $10$-node directed network in Fig.~\ref{fig:network}, where the only modification from the no-self-loop case is a negative self-loop of weight $-1$ added at node~9 (i.e., $A_{99}=-1$).}
\label{tab:combined_selfloop9}
\begin{minipage}[b]{0.48\textwidth}
\centering
\textbf{VCS}\\[0.5ex]
\begin{tabular}{c c c c c}
\toprule
\textbf{Node} & \textbf{$T=0.01$} & \textbf{$T=1$} & \textbf{$T=1000$} & \textbf{$T=10000$} \\
\midrule
1  & 0.1000 & 0.0997 & 0.0974 & 0.0974 \\
2  & 0.1000 & 0.1000 & 0.1020 & 0.1020 \\
3  & 0.1000 & 0.1000 & 0.1096 & 0.1096 \\
4  & 0.1000 & 0.1000 & 0.0874 & 0.0874 \\
5  & 0.1000 & 0.0997 & 0.0837 & 0.0837 \\
6  & 0.1000 & 0.0993 & 0.0606 & 0.0605 \\
7  & 0.1000 & 0.1013 & 0.2490 & 0.2492 \\
8  & 0.1000 & 0.0997 & 0.0419 & 0.0418 \\
9  & 0.1000 & 0.1003 & 0.1022 & 0.1022 \\
10 & 0.1000 & 0.1000 & 0.0661 & 0.0661 \\
\bottomrule
\end{tabular}
\end{minipage}%
\hfill
\begin{minipage}[b]{0.48\textwidth}
\centering
\textbf{AECS}\\[0.5ex]
\begin{tabular}{c c c c c}
\toprule
\textbf{Node} & \textbf{$T=0.01$} & \textbf{$T=1$} & \textbf{$T=1000$} & \textbf{$T=10000$} \\
\midrule
1  & 0.1000 & 0.1044 & 0.1269 & 0.1269 \\
2  & 0.1000 & 0.0955 & 0.0938 & 0.0938 \\
3  & 0.1000 & 0.0955 & 0.1001 & 0.1001 \\
4  & 0.1000 & 0.0955 & 0.0872 & 0.0872 \\
5  & 0.1000 & 0.0953 & 0.0739 & 0.0739 \\
6  & 0.1000 & 0.1044 & 0.1108 & 0.1108 \\
7  & 0.0999 & 0.0870 & 0.0763 & 0.0763 \\
8  & 0.1000 & 0.0953 & 0.0570 & 0.0569 \\
9  & 0.1003 & 0.1316 & 0.1953 & 0.1953 \\
10 & 0.1000 & 0.0955 & 0.0787 & 0.0787 \\
\bottomrule
\end{tabular}
\end{minipage}
\end{table*}

Table~\ref{tab:combined_selfloop9} shows that once a negative self-loop is added at node~9, the AECS assigns the largest weight to node~9 on long horizons; $p_9^{\rm AECS}\approx 0.195$ for $T=10000$, in sharp contrast to the no-self-loop case where the same node can be strongly downweighted.
This behavior is consistent with our theoretical analysis: when a source-like node has a negative self-loop (i.e., $\ell_i<0$), the  lower bound of AECS objective function $g_T$ no longer relaxes with the horizon, but instead remains bounded away from zero as $T\to\infty$; see~\eqref{nondecay}.
Consequently, driving the corresponding allocation $p_i$ to be very small would incur a large $g_T$, and the optimizer is discouraged from assigning a near-zero weight.

\section{Conclusion} \label{Sec:conclusion}

This paper clarified the structure and interpretation of controllability scoring for networked dynamical systems.
Under a node-wise virtual actuation model, the finite-time controllability Gramian decomposes additively across state nodes, yielding an affine matrix model that matches the information-matrix model in approximate OED; consequently, VCS and AECS correspond to the D- and A-optimal design criteria, and the classical D/A coordinate-invariance gap carries over under state-coordinate changes (VCS is invariant under any nonsingular transformation, whereas AECS is not in general).
We also emphasized a key distinction from approximate OED: while OED optima need not be unique, controllability scoring typically admits a unique optimal allocation, which is essential because the optimizer itself is interpreted as a node-level importance score.
Finally, we identified a long-horizon phenomenon with no OED analogue: AECS can strongly downweight source-like nodes when their self-dynamics are weak, whereas VCS need not due to its determinant-based coupling; when source-like nodes have negative self-loops, this tendency is mitigated and AECS can behave closer to VCS. Numerical examples corroborated these trends and showed that AECS and VCS can yield qualitatively different node scores on long horizons.

These observations also clarify the practical role of controllability
scoring in large-scale network intervention and actuator placement. When
exhaustive actuator-subset search is infeasible, VCS and AECS can serve
as continuous screening criteria for identifying promising intervention
sites, with VCS emphasizing reachable-volume expansion and AECS
emphasizing average-energy reduction. The long-horizon analysis further
suggests that source-like nodes need not always be prioritized by
energy-based criteria. Thus, in practical applications, the resulting
scores should be combined with additional considerations such as actuator
costs, sparsity, robustness, and domain-specific feasibility.

\appendix
\subsection{Technical lemmas for VCS and AECS lower bounds} \label{Ape}

This appendix collects two elementary matrix inequalities that we use to derive lower bounds for $f_T$ and $g_T$.
Lemma~\ref{lem:VCS} provides an upper bound on $\det W$.
Lemma~\ref{lem:inv_diag_bound} gives a convenient lower bound on the diagonal entries of $W^{-1}$.
These inequalities are standard; see, e.g., \cite{zhang2006schur}.
For the sake of a self-contained presentation, we include short proofs here.

\begin{lemma}\label{lem:VCS}
Let $W\in\mathbb{R}^{n\times n}$ be symmetric and positive definite.
Then, for each $i\in\{1,\dots,n\}$,
\begin{equation}\label{eq:det_upper_by_diagblock}
\det W
\le W_{ii}\,\det(W_{\bar i, \bar i})
\end{equation}
holds, where  $\bar i := \{1,\dots,n\}\setminus\{i\}$.
\end{lemma}
\begin{proof}
Let $\Pi$ be a permutation matrix that moves the $i$th coordinate to the first position, and set
$\widetilde W := \Pi W \Pi^\top$.
Then $\widetilde W\succ 0$, $\det \widetilde W=\det W$, and $\widetilde W$ admits the block partition
$\widetilde W=
\begin{bmatrix}
W_{ii} & W_{i, \bar i}\\
W_{\bar i, i} & W_{\bar i, \bar i}
\end{bmatrix}$.
By the Schur complement formula,
\begin{align}
    \det W
=\det \widetilde W
= W_{ii}\cdot \det \left(W_{\bar i, \bar i}-W_{\bar i,  i}W_{ii}^{-1}W_{i, \bar i}\right),
\end{align}
where the Schur complement
$W_{\bar i, \bar i}-W_{\bar i,  i}W_{ii}^{-1}W_{i, \bar i}\succ 0$.
Moreover, since $W_{\bar i,  i}W_{ii}^{-1}W_{i, \bar i}\succeq 0$, we have
$W_{\bar i, \bar i}-W_{\bar i,  i}W_{ii}^{-1}W_{i, \bar i}\preceq W_{\bar i, \bar i}$,
and hence \eqref{eq:det_upper_by_diagblock}
holds. 
\end{proof}

\begin{lemma}\label{lem:inv_diag_bound}
Let $W\in\mathbb{R}^{n\times n}$ be symmetric and positive definite.
Then, for each $i\in\{1,\dots,n\}$,
        $(W^{-1})_{ii}\ \ge\ \frac{1}{W_{ii}}$.
\end{lemma}
\begin{proof}
By the adjugate formula,
    \begin{align}
        (W^{-1})_{ii}=\frac{(\operatorname{adj}(W))_{ii}}{\det W}. \label{key_lower2}
    \end{align}
The $(i,i)$-entry of the adjugate matrix $\operatorname{adj}(W)$ equals the $(i,i)$-cofactor, and hence
    $(\operatorname{adj}(W))_{ii} = (-1)^{i+i}\det(W_{\bar i,\bar i})
    = \det(W_{\bar i,\bar i})$,
where $\bar i:=\{1,\dots,n\}\setminus\{i\}$.
Therefore, \eqref{key_lower2}
and Lemma~\ref{lem:VCS} yield the conclusion. 
\end{proof}


\bibliographystyle{IEEEtran}
\bibliography{main.bib}

\end{document}